\documentclass[oneside,english]{amsart}
\usepackage[T1]{fontenc}
\usepackage[latin9]{inputenc}
\usepackage{amstext}
\usepackage{amsthm}
\usepackage{amssymb}
\usepackage{wasysym}

\makeatletter
\numberwithin{equation}{section}
\numberwithin{figure}{section}
\theoremstyle{plain}
\newtheorem{thm}{\protect\theoremname}
\theoremstyle{plain}
\newtheorem{prop}[thm]{\protect\propositionname}
\theoremstyle{plain}
\newtheorem{lem}[thm]{\protect\lemmaname}
\theoremstyle{plain}
\newtheorem{cor}[thm]{\protect\corollaryname}
\theoremstyle{remark}
\newtheorem{rem}[thm]{\protect\remarkname}

\numberwithin{thm}{section}
\usepackage{xcolor}

\makeatother

\usepackage{babel}
\providecommand{\corollaryname}{Corollary}
\providecommand{\lemmaname}{Lemma}
\providecommand{\propositionname}{Proposition}
\providecommand{\remarkname}{Remark}
\providecommand{\theoremname}{Theorem}

\begin{document}
\title{On the support of free convolutions}
\author{Serban Belinschi, Hari Bercovici, and Ching-Wei Ho}
\address{Institut de Math\'ematiques de Toulouse: UMR5219, Universit� de Toulouse,
CNRS; UPS , F-31062 Toulouse, France}
\email{Serban.Belinschi@math.univ-toulouse.fr}
\address{Mathematics Department, Indiana University, Bloomington, IN 47405,
USA}
\email{bercovic@indiana.edu}
\address{Institute of Mathematics, Academia Sinica, Taipei 10617, Taiwan}
\email{chwho@gate.sinica.edu.tw}
\subjclass[2000]{Primary: 46L54 Secondary: 30D05}
\begin{abstract}
We extend to arbitrary measures results of Bao, Erd\"os, Schnelli,
Moreillon, and Ji on the connectedness of the supports of additive
convolutions of measures on $\mathbb{R}$ and of free multiplicative
convolutions of measures on $\mathbb{R}_{+}$. More precisely, the
convolution of two measures with connected supports also has connected
support. The result holds without any absolute continuity or bounded
support hypotheses on the measures being convolved. We also show that
the results of Moreillon and Schnelli concerning the number of components
of the support of a free additive convolution hold for arbitrary measures
with bounded supports. Finally, we provide an approach to the corresponding
results in the case of free multiplicative convolutions of probability
measures on the unit circle.
\end{abstract}

\maketitle

\section{Introduction\label{sec:Introduction}}

Suppose that $\mu_{1}$ and $\mu_{2}$ are Borel probability measures
on $\mathbb{R}$ whose supports ${\rm supp}(\mu_{1})$ and ${\rm supp}(\mu_{2})$
have $n_{1}$ and $n_{2}$ connected components, respectively, where
$n_{1},n_{2}<+\infty$. The question of whether the free convolution
$\mu=\mu_{1}\boxplus\mu_{2}$ is such that ${\rm supp}(\mu)$ has
a finite number $n$ of connected components, and possible estimates
of $n$, was investigated in two recent works \cite{bao-etc,mor-sch}.
These authors consider compactly supported and absolutely continuous
measures $\mu_{j}$ whose densities satisfy a Jacobi type condition.
More precisely, these densities are comparable to a function of the
type
\[
\rho(t)=(t-a)^{\alpha}(b-t)^{\beta}
\]
on each component $[a,b]$ of the support, where $\alpha,\beta\in(-1,1)$.
Under these conditions, it is shown in \cite{mor-sch} (among other
results, including better estimates in some cases) that 
\[
n\le2n_{1}n_{2}-1.
\]
In the special case $n_{1}=n_{2}=1$, we obtain $n=1$, and this result
appears in the earlier work \cite{bao-etc}.

The authors of \cite{mor-sch} conjecture that this estimate is still
true for other Jacobi type densities, for instance when $\alpha,\beta\in(-1,\infty)$.
Using a simple approximation argument, we show that the inequality
$n\le2n_{1}n_{2}-1$ remains true if $\mu_{1}$ and $\mu_{2}$ are
only assumed to have compact supports. This suggests that a direct,
and possibly simpler, argument may exist, not using Jacobi type measures.
We develop here a few basic facts that would be useful for such an
argument, and we also extend the case in which $n_{1}=n_{2}=1$ to
measures whose supports are not necessarily compact. We hope that
these observations will be useful in eventually finding better upper
bounds for $n$ when $n_{1},n_{2}>1$. (Note that the inequality $n\le n_{1}n_{2}$
holds for classical convolutions, and equality is possible.)

The work of \cite{bao-etc} was extended by Ji \cite{Ji} to free
multiplicative convolutions of probability measures on $[0,+\infty)$.
In particular, suppose that $\mu_{1}$ and $\mu_{2}$ are two such
measures that have compact, connected support, and satisfy a Jacobi
condition. Then it is shown in \cite{Ji} that $\mu_{1}\boxtimes\mu_{2}$
has connected support. We extend this result to arbitrary measures
$\mu_{1},\mu_{2}$ that have connected, but not necessarily bounded,
support. Finally, we describe the analytic framework that might allow
for proving analogous results for free multiplicative convolutions
of measures on the unit circle.

The paper is organized as follows. In Section \ref{sec:An-approximation-argument}
we argue, using elementary spectral theory, that the estimate $n<2n_{1}n_{2}$
proved in \cite{mor-sch} does in fact hold for arbitrary measures
with bounded support. The same argument can be used to extend the
connectedness result of \cite{Ji} to arbitrary measures with bounded
support in $[0,+\infty)$. The argument is essentially the same as
in the additive case and the details are not included. In Section
\ref{sec:prelim} we review a result due to Lehner \cite{lehner}.
This result is used subsequently to identify the boundary points of
the support of a free convolution. Our approach to the identification
of the support of free convolutions is presented in Sections \ref{sec:additive},
\ref{sec:mult-on-R}, and \ref{sec:multiplicative-T} for free additive
convolution, free multiplicative convolution on $\mathbb{R}_{+}$,
and free multiplicative convolution on the unit circle $\mathbb{T}$,
respectively. Naturally, there is common ground between our arguments
and earlier work on the subject, though we avoid some of the technicalities
that are necessary when considering Jacobi type measures.

\section{An approximation argument\label{sec:An-approximation-argument}}

The following result follows from the upper continuity of the spectrum
combined with a result of Newburgh \cite{newb} (see also \cite{aupetit}).
\begin{prop}
\label{prop:number of components for nearby elements}Suppose that
$\mathcal{A}$ is a Banach algebra and $x\in\mathcal{A}$ has the
property that the spectrum $\sigma(x)$ has a finite number $n$ of
connected components. There exists $\varepsilon>0$ such that every
$y\in\mathcal{A}$ that satisfies $\|x-y\|<\varepsilon$ is such that
$\sigma(y)$ has at least $n$ connected components. 
\end{prop}

Indeed, suppose that $U_{1},\dots,U_{n}$ are pairwise disjoint connected
open sets in $\mathbb{C}$ such that $U=\bigcup_{j=1}^{n}U_{j}$ contains
$\sigma(x)$ and each $U_{j}$ contains a connected component of $\sigma(x)$.
Then, for small $\varepsilon,$ the inequality $\|x-y\|<\varepsilon$
implies that $\sigma(y)\subset U$, while Newburgh's theorem \cite[Theorem 3.4.4]{aupetit}
implies that $\sigma(y)\cap U_{j}\ne\varnothing$, $j=1,\dots,n,$
for sufficiently small $\varepsilon$.

For the following observation, we use the space $L^{\infty}$ corresponding
to the interval $(0,1)$ endowed with the usual Lebesgue measure.
(The square root in (3) can be replaced by any exponent in $(-1,+\infty)$.)
\begin{lem}
\label{lem:Jacobi approximation} Suppose that $\mu$ is a Borel probability
measure with compact support and that ${\rm supp}(\mu)$ has a finite
number $n$ of connected components. Then, for every $\varepsilon>0$,
there exist $x,y\in L^{\infty}$ such that
\begin{enumerate}
\item $\|x-y\|_{\infty}<\varepsilon,$ 
\item the distribution of x is $\mu$,
\item the distribution $\nu$ of $y$ is supported on a union of $n$ disjoint
intervals $[a_{j},b_{j}]$, $j=1,\dots,n$, it is absolutely continuous,
and its density satisfies
\[
c_{j}\sqrt{(t-a_{j})(b_{j}-t)}\le\rho(t)\le C_{j}\sqrt{(t-a_{j})(b_{j}-t)},\quad t\in[a_{j},b_{j}],
\]
for some $c_{j},C_{j}>0$.
\end{enumerate}
\end{lem}

\begin{proof}
We first construct the distribution $\nu$. Denote by
\[
F(t)=\mu((-\infty,t]),\quad t\in\mathbb{R},
\]
the distribution function of $\mu$. Choose pairwise disjoint intervals
$[a_{j},b_{j}]$, each one containing one of the connected components
of the support of $\mu$. Thus, $F$ is locally constant on $\mathbb{R}\backslash\bigcup_{j=1}^{n}[a_{j},b_{j}]$.
We construct $\nu$ such that the distribution function
\[
G(t)=\nu((-\infty,t]),\quad t\in\mathbb{R},
\]
equals $F(t)$ for $t\notin\bigcup_{j=1}^{n}[a_{j},b_{j}]$, is differentiable
on $\mathbb{R}$, strictly increasing on each $(a_{j},b_{j})$, $G(t)-G(a_{j})=(t-a_{j})^{3/2}$
for $t>a_{j}$ close to $a_{j}$, and $G(b_{j})-G(t)=(b_{j}-t)^{3/2}$
for $t<b_{j}$ close to $b_{j}$. We also require that there is some
$h\in(0,\varepsilon)$ such that $G(kh)=F(kh)$ for every $k\in\mathbb{Z}.$
The functions $x$ and $y$ are then defined essentially as the inverse
functions of $F$ and $G$ respectively, in the sense that
\[
x(s)=\inf\{t\in\mathbb{R}:F(t)>s\},\quad s\in(0,1),
\]
with an analogous definition for $y$. It is now easy to verify that
the functions $x,y\in L^{\infty}$ satisfy the requirements of the
lemma. 
\end{proof}
Note, incidentally, that $x$ depends solely on $\mu$, so increasingly
precise approximations $y$ can be constructed without changing $x$.
Therefore, using the free product $L^{\infty}\star L^{\infty}$ (relative
to the standard expected values on the two algebras) easily yields
the following result.
\begin{cor}
\label{cor:Jacoby approx for free var}Suppose that $\mu_{1}$ and
$\mu_{2}$ are two compactly supported measures on $\mathbb{R}$ such
that the supports ${\rm supp}(\mu_{1})$ and ${\rm supp}(\mu_{2})$
have $n_{1}$ and $n_{2}$ connected components, with $n_{1},n_{2}<+\infty$.
There exist families $\mathcal{F}_{j}=\{x_{j}(\varepsilon\}:\varepsilon>0\}$,
$j=1,2$, of selfadjoint random variables such that
\begin{enumerate}
\item $\mathcal{F}_{1}$ is free from $\mathcal{F}_{2}$,
\item The distribution of $x_{j}(0)$ is $\mu_{j}$, $j=1,2$.
\item $\lim_{\varepsilon\to0}\|x_{j}(\varepsilon)-x_{j}(0)\|=0$.
\item for $\varepsilon>0$ and $j=1,2$, the distribution $\nu_{j}(\varepsilon)$
of $x_{j}(\varepsilon)$ is absolutely continuous, supported on $n_{j}$
pairwise disjoint compact intervals $[a,b]$, and the density of $\nu_{j}(\varepsilon)$
is comparable with $\sqrt{(t-a)(b-t)}$ on each of these intervals.
\end{enumerate}
\end{cor}

With the notation of the preceding corollary, we have
\[
\lim_{\varepsilon\to0}\|(x_{1}(0)+x_{2}(0))-(x_{1}(\varepsilon)+x_{2}(\varepsilon))\|=0,
\]
and therefore, for sufficiently small $\varepsilon$, $\sigma(x_{1}(\varepsilon)+x_{2}(\varepsilon))$
has at least as many connected components as $\sigma(x_{1}+x_{2})$
by Proposition \ref{prop:number of components for nearby elements}.
The distribution of $x_{1}+x_{2}$ is precisely $\mu_{1}\boxplus\mu_{2}$,
so $\sigma(x_{1}+x_{2})={\rm supp}(\mu_{1}\boxplus\mu_{2})$. Similarly,
$\sigma(x_{1}(\varepsilon)+x_{2}(\varepsilon))={\rm supp}(\nu_{1}(\varepsilon)\boxplus\nu_{2}(\varepsilon))$,
and the number of connected components of this set was estimated by
the results of \cite{mor-sch}. This proves the following result.
\begin{cor}
\label{cor:Moreillon estimate}Under the hypotheses of Corollary \emph{\ref{cor:Jacoby approx for free var}},
${\rm supp}(\mu_{1}\boxplus\mu_{2})$ has at most $2n_{1}n_{2}-1$
connected components.
\end{cor}

Naturally, under this generality, some connected components of $\mu_{1}\boxplus\mu_{2}$
can be singletons, and there can even be point masses embedded in
the absolutely continuous support. Such situations do not arise for
the measures considered in \cite{mor-sch}. However, many of the sharper
estimates of that work should survive in complete generality.

\section{Invertibility and $L^{2}$ norms\label{sec:prelim}}

We consider a tracial $W^{*}$-probability space $(\mathcal{A},\tau)$,
that is, $\mathcal{A}$ is a von Neumann algebra and $\tau$ is a
faithful normal trace on $\mathcal{A}$ such that $\tau(1_{\mathcal{A}})=1$.
Thus, the elements of $\mathcal{A}$ play the role of bounded random
variables. Arbitrary random variables are elements of the larger algebra
$\widetilde{\mathcal{A}}$ consisting of the (generally unbounded)
operators affiliated with $\mathcal{A}$. The trace $\tau$ extends
to the positive operators in $\widetilde{\mathcal{A}}$, thus allowing
us to talk of the $L^{2}$ norm of an arbitrary random variable: $\|a\|_{2}^{2}=\tau(a^{*}a)$.
The Hilbert space $L^{2}(\tau)$ consists of those $a\in\widetilde{\mathcal{A}}$
for which this norm is finite.

Suppose that $a_{1},a_{2}\in\mathcal{A}$ are two $*$-free random
variables that satisfy $\tau(a_{1})=\tau(a_{2})=0$. It was observed
in \cite{haagerup-larsen} that the spectral radius of the product
$a_{1}a_{2}$ is $\|a_{1}\|_{2}\|a_{2}\|_{2}$ and
\[
\|(\lambda1_{\mathcal{A}}-a_{1}a_{2})^{-1}\|_{2}^{2}=\frac{1}{|\lambda|^{2}-\|a_{1}\|_{2}^{2}\|a_{2}\|_{2}^{2}},\quad\lambda\in\mathbb{C},|\lambda|>\|a_{1}\|_{2}\|a_{2}\|_{2}.
\]
This last formula arises from the equality
\[
(\lambda1_{\mathcal{A}}-a_{1}a_{2})^{-1}=\sum_{n=0}^{\infty}\frac{1}{\lambda^{n+1}}(a_{1}a_{2})^{n},
\]
and from the fact that the terms of the series form an orthogonal
sequence in $L^{2}(\mathcal{A})$ with
\[
\|(a_{1}a_{2})^{n}\|_{2}=\|a_{1}\|_{2}^{n}\|a_{2}\|_{2}^{n}.
\]
We reproduce the following useful result of Lehner \cite{lehner}.
\begin{lem}
\label{lem:Lehner} Suppose that $(a_{n})_{n\in\mathbb{N}}$ and $(b_{n})_{n\in\mathbb{N}}$
are two sequences in $\mathcal{A}$ that converge in norm to $a$
and $b$ respectively. Suppose also that, for every $n\in\mathbb{N},$
the variables $a_{n}$ and $b_{n}$ are $*$-free, $\tau(a_{n})=\tau(b_{n})=0$,
and $\|a_{n}\|_{2}\|b_{n}\|_{2}<1$. Then $1_{\mathcal{A}}-ab$ is
invertible in $\mathcal{A}$ if and only if $\|a\|_{2}\|b\|_{2}<1$.
\end{lem}

\begin{proof}
Observe that $a$ and $b$ are also $*$-free and $\tau(a)=\tau(b)=0$.
If $\|a\|_{2}\|b\|_{2}<1$, the spectral radius of $ab$ is less than
one, and thus $1_{\mathcal{A}}-ab$ is invertible. On the other hand,
if $\|a\|_{2}\|b\|_{2}\ge1$, the observations preceding the statement
of the lemma imply that
\[
\lim_{n\to\infty}\|(1_{\mathcal{A}}-a_{n}b_{n})^{-1}\|_{2}=+\infty,
\]
and therefore 
\[
\lim_{n\to\infty}\|(1_{\mathcal{A}}-a_{n}b_{n})^{-1}\|=+\infty.
\]
Since $1_{\mathcal{A}}-a_{n}b_{n}$ converges to $1_{\mathcal{A}}-ab$
in norm, and since the map $x\mapsto x^{-1}$ is norm-continuous,
it follows that $1_{\mathcal{A}}-ab$ is not invertible.
\end{proof}
Of course, the preceding lemma is true if ordinary sequences are replaced
by nets, so the following corollary is immediate. (However, we only
use this result when $X$ is a subset of the Riemann sphere.)
\begin{cor}
\label{cor: clopen}Suppose that $X$ is a topological space and that
$u_{1},u_{2}:X\to\mathcal{A}$ are continuous\emph{ }functions\emph{
(}when $\mathcal{A}$ is endowed with the operator norm\emph{)} such
that, for every $x\in X$, $u_{1}(x)$ and $u_{2}(x)$ are $*$-free,
$\tau(u_{1}(x))=\tau(u_{2}(x))=0$, and $1_{\mathcal{A}}-u_{1}(x)u_{2}(x)$
is invertible. Then the set
\[
Y=\{x\in X:\|u_{1}(x)\|_{2}\|u_{2}(x)\|_{2}<1\}
\]
is both open and closed in $X$. In particular, if $X$ is connected,
we have either $Y=X$ or $Y=\varnothing$.
\end{cor}

\begin{rem}
\label{rem:a bit of algebra}The preceding result will be used, for
instance, when $a_{j}=b_{j}^{-1}-1_{\mathcal{A}}$, where $b_{1},b_{2}\in\widetilde{\mathcal{A}}$
are two boundedly invertible operators such that $\tau(b_{1}^{-1})=\tau(b_{2}^{-1})=1$.
In this special case, the identity
\[
1_{\mathcal{A}}-a_{1}a_{2}=b_{1}^{-1}(1_{\mathcal{A}}-b_{1}-b_{2})b_{2}^{-1},
\]
and therefore the invertibility of $1_{\mathcal{A}}-a_{1}a_{2}$ is
equivalent to the invertibility of $1_{\mathcal{A}}-b_{1}-b_{2}$.
(This identity was used by \cite{haagerup-fields} to provide a proof
of the linearization property of the $R$-transform.)
\end{rem}

Suppose now that $x_{1},x_{2}\in\widetilde{\mathcal{A}}$ are selfadjoint
operators. We denote by $\mathbb{H}=\{\alpha+i\beta:\alpha,\beta\in\mathbb{R},\beta>0\}$
the complex upper half-plane. We need Remark \ref{rem:a bit of algebra}
for elements of the form 
\[
b_{j}=\tau((\lambda_{j}1_{\mathcal{A}}-x_{j})^{-1})(\lambda_{j}1_{\mathcal{A}}-x_{j}),\quad\lambda_{j}\in\mathbb{C}\backslash\sigma(x_{j}),j=1,2,
\]
where $\tau((\lambda_{j}1_{\mathcal{A}}-a_{j})^{-1})\ne0$. In this
case
\[
a_{j}=\frac{1}{\tau((\lambda_{j}1_{\mathcal{A}}-x_{j})^{-1})}(\lambda_{j}1_{\mathcal{A}}-x_{j})^{-1}-1_{\mathcal{A}},\quad j=1,2.
\]
In order to apply Lemma \ref{lem:Lehner} to such elements, we need
to estimate their $2$-norms. In the following result we use the notation
$\lambda\to_{\varangle}\infty$ to indicate that $\lambda\in\mathbb{H}$
tends to infinity but the ratio $\Re\lambda/\Im\lambda$ remains bounded.
In other words, $\lambda$ tends to $\infty$ nontangentially (to
$\mathbb{R})$.
\begin{lem}
\label{lem:2-norm of centered resolvents} Fix a selfadjoint element
$x\in\widetilde{\mathcal{A}}$ and set
\[
a(\lambda)=\frac{1}{\tau((\lambda1_{\mathcal{A}}-a)^{-1})}(\lambda1_{\mathcal{A}}-a)^{-1}-1_{\mathcal{A}},\quad\lambda\in\mathbb{H}.
\]
Then
\[
\lim_{\lambda\to_{\varangle}\infty}\|a(\lambda)\|_{2}=0.
\]
\end{lem}

\begin{proof}
Denote by $\mu$ the distribution of $x$. That is, $\mu$ is a Borel
probability measure on $\mathbb{R}$ such that
\[
\tau(u(x))=\int_{\mathbb{R}}u(t)\,d\mu(t)
\]
for every bounded Borel function $u$ on $\mathbb{R}.$ In particular,
\[
\tau((\lambda1_{\mathcal{A}}-a)^{-1})=\int_{\mathbb{R}}\frac{d\mu(t)}{\lambda-t}=\frac{1}{\lambda}+\frac{1}{\lambda}\int_{\mathbb{R}}\frac{t\,d\mu(t)}{\lambda-t},\quad\lambda\in\mathbb{H}.
\]
This implies that
\begin{equation}
\tau((\lambda1_{\mathcal{A}}-a)^{-1})=\frac{1}{\lambda}(1+o(1))\text{ as }\lambda\to_{\varangle}\infty.\label{eq:*}
\end{equation}
 Since
\[
a(\lambda)=\frac{1}{\tau((\lambda1_{\mathcal{A}}-a)^{-1})}[(\lambda1_{\mathcal{A}}-a)^{-1}-\tau((\lambda1_{\mathcal{A}}-a)^{-1})],
\]
we have
\[
\|a(\lambda)\|_{2}=\frac{1}{|\tau((\lambda1_{\mathcal{A}}-a)^{-1})|}\|(\lambda1_{\mathcal{A}}-a)^{-1}-\tau((\lambda1_{\mathcal{A}}-a)^{-1})\|_{2},
\]
and therefore the norm on the right is precisely the standard deviation
of the function $1/(\lambda-t)$ viewed as a random variable in $L^{2}(\mu)$.
Thus by (\ref{eq:*}),
\begin{align*}
\|(\lambda1_{\mathcal{A}}-a)^{-1}-\tau((\lambda1_{\mathcal{A}}-a)^{-1})\|_{2}^{2} & \le\|(\lambda1_{\mathcal{A}}-a)^{-1}-\lambda^{-1}1_{\mathcal{A}}\|_{2}^{2}+o(|\lambda|^{-1})\\
 & =\int_{\mathbb{R}}\left|\frac{1}{\lambda-t}-\frac{1}{\lambda}\right|^{2}d\mu(t)+o(|\lambda|^{-1})\\
 & =\frac{1}{|\lambda|^{2}}\int_{\mathbb{R}}\left|\frac{t}{\lambda-t}\right|^{2}d\mu(t)+o(|\lambda|^{-1}),
\end{align*}
and this implies (using, for instance, the Lebesgue bounded convergence
theorem) that
\begin{equation}
\|(\lambda1_{\mathcal{A}}-a)^{-1}-\tau((\lambda1_{\mathcal{A}}-a)^{-1})\|_{2}=o(|\lambda|^{-1})\text{ as }\text{\ensuremath{\lambda\to_{\varangle}\infty.}}\label{eq:**}
\end{equation}
The lemma follows.
\end{proof}

\section{\label{sec:additive} The analysis of free additive convolution}

Suppose that $\mu$ is a Borel probability measure on $\mathbb{R}$.
The \emph{Cauchy transform}
\[
G_{\mu}(z)=\int_{\mathbb{R}}\frac{d\mu(t)}{z-t},\quad z\in\mathbb{C}\backslash\mathbb{R},
\]
is an analytic function that maps the upper half-plane $\mathbb{H}$
to $-\mathbb{H}$, $G_{\mu}(\overline{z})=\overline{G_{\mu}(z)}$
for $z\in\mathbb{H}$, and
\[
\lim_{y\uparrow\infty}iyG_{\mu}(iy)=1.
\]
 Moreover, every analytic function $G:\mathbb{C}\backslash\mathbb{R}\to\mathbb{C}$
that satisfies these conditions is the Cauchy transform of a unique
Borel probability measure on $\mathbb{R}$ (see \cite{akh}). We also
note that $G_{\mu}$ is in fact analytic in the complement $\mathbb{C}\backslash{\rm supp}(\mu)$
of the support of $\mu,$ and real-valued on $\mathbb{R}\backslash{\rm supp}(\mu)$.
It follows that the \emph{reciprocal Cauchy transform}
\[
F_{\mu}(z)=\frac{1}{G_{\mu}(z)}
\]
is a meromorphic function in $\mathbb{C}\backslash{\rm supp}(\mu)$.
Moreover, $F_{\mu}$ maps $\mathbb{H}$ to $\mathbb{H},$ and
\[
\lim_{y\uparrow\infty}\frac{F_{\mu}(iy)}{iy}=1.
\]
 Conversely, every analytic function $F:\mathbb{C}\backslash\mathbb{R}\to\mathbb{C}$
such that $F(\mathbb{H})\subset\mathbb{H},$ $F(\overline{z})=\overline{F(z)}$,
and
\[
\lim_{y\uparrow\infty}\frac{F(iy)}{iy}=1,
\]
is the reciprocal Cauchy transform of a unique Borel probability measure
on $\mathbb{R}$ \cite{akh}. There is an alternate description, due
to R. Nevanlinna, of such reciprocal Cauchy transforms $F$: there
exist a finite, positive Borel measure $\sigma$ on $\mathbb{R}$,
and a constant $c\in\mathbb{R},$ such that
\begin{equation}
F(z)=c+z+\int_{\mathbb{R}}\frac{1+zt}{t-z}\,d\sigma(t),\quad z\in\mathbb{C}\backslash\mathbb{R}.\label{eq:Nev-rep general F}
\end{equation}
In case $F=F_{\mu},$ we denote the corresponding measure $\sigma$
by $\sigma_{\mu}$ and the constant $c$ by $c_{\mu}.$ The measure
$\sigma_{\mu}$ is zero precisely when $\mu$ is a unit point mass.
In general, 
\[
{\rm supp}(\sigma_{\mu})=({\rm supp}(\mu)\backslash A)\cup B,
\]
where $A$ is the set of isolated point masses of $\mu$ and 
\[
B=\{t\in\mathbb{R\backslash{\rm supp}}(\mu):G_{\mu}(t)=0\}
\]
consists of the poles of $F_{\mu}$ in $\mathbb{C}\backslash{\rm supp}(\mu)$.
There is at most one such pole in each bounded connected component
of $\mathbb{R}\backslash{\rm supp}(\mu)$. It is easily seen from
(\ref{eq:Nev-rep general F}) that $\Im F_{\mu}(z)\ge\Im z$ for every
$z\in\mathbb{H},$ and the inequality is strict unless $\mu$ is a
point mass. Moreover, $F_{\mu}$ is real-valued and strictly increasing
on every connected component of $\mathbb{R}\backslash{\rm supp}(\sigma_{\mu})$.

Next, we discuss the free additive convolution. Suppose that $\mu_{1}$
and $\mu_{2}$ are two Borel probability measures on $\mathbb{R}.$
The free additive convolution $\mu=\mu_{1}\boxplus\mu_{2}$ is a new
Borel probability measure on $\mathbb{R},$ first defined in \cite{voic-boxplus}
(see also \cite{maass,BV-unbounded}). It is useful for our purposes
to use the fact that  there exists an analytic subordination relation
between the functions $F_{\mu}$, $F_{\mu_{1}},$ and $F_{\mu_{2}}$,
first seen in \cite{voic-fish1} and \cite{biane} (see also \cite{voic-coalg}
for a more general result). The following result incorporates more
recent information found in \cite{serb-boundedness,serb-leb,bb-new-approach,BBH}
If $\varphi:\mathbb{H}\to\mathbb{H}$ is analytic, different from
the identity map, then either $\varphi$ has a unique fixed point
$z_{0}\in\mathbb{H},$ or the iterates 
\[
\varphi^{n}=\underbrace{\varphi\circ\cdots\circ\varphi}_{n\text{ times}}
\]
 converge everywhere to a point $z_{0}\in\mathbb{R}\cup\{\infty\}$.
In both cases, we call $z_{0}$ the \emph{Denjoy-Wolff point} of $\varphi$.
(This may be at variance with the common use of this term when $\varphi$
is a conformal automorphism.)
\begin{thm}
\label{thm:boxplus subordination}Let $\mu_{1}$ and $\mu_{2}$ be
Borel probability measures on $\mathbb{R}$, neither of which is a
point mass, and let $\mu=\mu_{1}\boxplus\mu_{2}$. There exist unique
continous functions $\omega_{1},\omega_{2}:\mathbb{H}\cup\mathbb{R}\to\mathbb{H}\cup\mathbb{R}\cup\{\infty\}$
such that\emph{:}
\begin{enumerate}
\item $\omega_{j}(\mathbb{H})\subset\mathbb{H}$ and $\omega_{j}$ is analytic
on $\mathbb{H}$ for $j=1,2$.
\item $F_{\mu}(z)=F_{\mu_{1}}(\omega_{1}(z))=F_{\mu_{2}}(\omega_{2}(z))=\omega_{1}(z)+\omega_{2}(z)-z$
for every $z\in\mathbb{H}.$
\item $\lim_{y\uparrow\infty}\omega_{j}(iy)/iy=1$ for $j=1,2$, and
\item For every $\alpha\in\mathbb{H\cup\mathbb{R}}$, with only one possible
exception, $\omega_{1}(\alpha)$ is the Denjoy-Wolff point of the
map $\varphi_{\alpha}(z)=\alpha+h_{2}(\alpha+h_{1}(z))$, where $h_{j}(z)=F_{\mu_{j}}(z)-z$
for $j=1,2$ and $z\in\mathbb{H}$. Symmetrically, $\omega_{2}(\alpha)$
is the Denjoy-Wolff point of the map $z\mapsto\alpha+h_{1}(\alpha+h_{2}(z))$.
In the exceptional case, $\alpha$ is real, $\varphi_{\alpha}$ is
the identity map, and this can occur only when the supports of both
$\mu_{1}$ and $\mu_{2}$ consist of two points.
\end{enumerate}
\end{thm}

An immediate consequence of this result is that the reciprocal Cauchy
transform $F_{\mu}$ also extends continuously to a map from $\mathbb{H}\cup\mathbb{R}$
to $\mathbb{H}\cup\mathbb{R}\cup\{\infty\}$. We also note that 
\[
\omega_{j}(z)=\gamma_{j}+z+\int_{\mathbb{R}}\frac{1+zt}{t-z}\,d\rho_{j}(t),\quad z\in\mathbb{H},j=1,2,
\]
where $\gamma_{j}\in\mathbb{R}$ and $\sigma_{j}$ is a finite, positive
Borel measure on $\mathbb{R}.$ The identity in part (2) of the theorem
yields the equalities
\begin{equation}
c_{\mu}=\gamma_{1}+\gamma_{2},\quad\sigma_{\mu}=\rho_{1}+\rho_{2}.\label{c=00003Dgamma1+gamma2}
\end{equation}
 In particular, $\sigma_{\mu}(\{t\})=\rho_{1}(\{t\})+\rho_{2}(\{t\})$
for every pole $t$ of $F_{\mu}$.

A calculation equivalent to the following lemma is found in \cite{lehner}.
\begin{lem}
\label{lem:F'-1=00003DL2norm}Suppose that $x\in\widetilde{\mathcal{A}}$
is selfadjoint, $\nu$ is the distribution of $x$, and $t\in\mathbb{R}\backslash{\rm supp}(\nu)$
is such that $G_{\nu}(t)\ne0$. Define $b=G_{\nu}(t)(t1_{\mathcal{A}}-x)$
and $a=b^{-1}-1_{\mathcal{A}}$. Then $\|a\|_{2}^{2}=F_{\nu}'(t)-1$.
\end{lem}

\begin{proof}
We have
\begin{align*}
\|a\|_{2}^{2} & =\int_{\mathbb{R}}\left|\frac{1}{G_{\nu}(t)(t-s)}-1\right|^{2}d\nu(s)\\
 & =\frac{1}{G_{\nu}(t)^{2}}\int_{\mathbb{R}}\left|\frac{1}{t-s}-\int_{\mathbb{R}}\frac{1}{t-\tau}d\nu(\tau)\right|^{2}d\nu(s),
\end{align*}
and the last integral represents the variance of the random variable
$s\mapsto1/(t-s)$ in $L^{2}(\nu)$. Now,
\begin{align*}
F_{\nu}'(t)-1 & =\frac{-G_{\nu}'(t)}{G_{\nu}(t)^{2}}-1\\
 & =\frac{1}{G_{\nu}(t)^{2}}\left[\int_{\mathbb{R}}\frac{1}{(t-s)^{2}}d\nu(s)-\left[\int_{\mathbb{R}}\frac{1}{t-s}d\nu(s)\right]^{2}\right],
\end{align*}
and this is again $1/G_{\nu}(t)^{2}$ times the variance of $s\mapsto1/(t-s)$.
\end{proof}
In the following result and in its proof we use the notation established
before for subordination functions and their Nevanlinna representations.
\begin{prop}
\label{pro:omega maps resolvent to res} Let $\mu_{1}$ and $\mu_{2}$
be Borel probability measures on $\mathbb{R}$, and set $\mu=\mu_{1}\boxplus\mu_{2}$.
Suppose that $J=(\alpha,\beta)\subset\mathbb{R}$ is disjoint from
${\rm supp}(\mu)$ and that $F_{\mu}$ is finite at every point of
$J$. Then, for $j=1,2$, we have\emph{:}
\begin{enumerate}
\item the function $\omega_{j}$ is finite and real valued on $J$,  $\omega_{j}(J)$
is disjoint from ${\rm supp}(\mu_{j})$, and
\item $(F_{\mu_{1}}'(\omega_{1}(t))-1)(F'_{\mu_{2}}(\omega_{2}(t))-1)<1$
for $t\in J$.
\end{enumerate}
\end{prop}

\begin{proof}
Since $F_{\mu}$ is real-valued on $J$, we have $\sigma_{\mu}(J)=0$.
The second equality in (\ref{c=00003Dgamma1+gamma2}) shows that $\rho_{1}(J)=\rho_{2}(J)=0$
as well, and thus $\omega_{1}$ and $\omega_{2}$ are finite and real-valued
on $J$. Since $\omega'_{j}(t)\ge1$ for $t\in J$, $\omega_{j}$
maps $J$ bijectively to an open interval $J_{j}\subset\mathbb{R}$
and it extends to a conformal map of some open neighborhood of $J$
onto an open neighborhood of $J_{j}$. The subordination equation
\[
F_{\mu}(\lambda)=F_{\mu_{j}}(\omega_{j}(\lambda)),\quad j=1,2,\ \lambda\in\mathbb{H},
\]
implies that $F_{\mu_{j}}$ extends continuously to the interval $J_{j}$
and the extension is real-valued and finite on that interval. Therefore,
\[
\mu_{j}(J_{j})=\sigma_{\mu_{j}}(J_{j})=0,
\]
thus concluding the proof of (1).

To prove (2), it is convenient to consider two freely independent
(generally unbounded) selfadjoint operators $x_{1}$ and $x_{2},$
affiliated with $(\mathcal{A},\tau),$ whose distributions are $\mu_{1}$
and $\mu_{2}$, respectively. Then $x=x_{1}+x_{2}$ has distribution
$\mu$. For $\lambda\in\mathbb{H}\cup J$, we set
\[
b_{j}(\lambda)=G_{\mu}(\lambda)(\omega_{j}(\lambda)1_{\mathcal{A}}-x_{j}),\ a_{j}(\lambda)=b_{j}(\lambda)^{-1}-1_{\mathcal{A}},\quad j=1,2.
\]
We observe that
\[
G_{\mu}(\lambda)=G_{\mu_{j}}(\omega_{j}(\lambda))=\tau((\omega_{j}(\lambda)1_{\mathcal{A}}-x_{j})^{-1}),\quad j=1,2,\ \lambda\in\mathbb{H}\cup J,
\]
and this places us in the context of Remark \ref{rem:a bit of algebra},
that is, $\tau(a_{j}(\lambda))=0$. Observe that
\begin{align*}
b_{1}(\lambda)+b_{2}(\lambda)-1 & =G_{\mu}(\lambda)\left[\left(\omega_{1}(\lambda)+\omega_{2}(\lambda)-\frac{1}{G_{\mu}(\lambda)}\right)1_{\mathcal{A}}-x_{1}-x_{2}\right]\\
 & =G_{\mu}(\lambda)(\lambda1_{\mathcal{A}}-x_{1}-x_{2})
\end{align*}
is invertible for $\lambda\in\mathbb{H}\cup J$, and thus $1-a_{1}(\lambda)a_{2}(\lambda)$
is also invertible, for every $\lambda\in\mathbb{H}\cup J$. Moreover,
$\|a_{j}(is)\|_{2}<1$ for large $s>0$. This follows from Lemma \ref{lem:2-norm of centered resolvents}
because $\omega_{j}(iy)$ tends to $\infty$ nontangentially as $y\uparrow+\infty$.
Since the maps $\lambda\mapsto a_{j}(\lambda)$ are continuous, Corollary
\ref{cor: clopen} implies that
\[
\|a_{1}(\lambda)\|_{2}\|a_{2}(\lambda)\|_{2}<1,\quad\lambda\in\mathbb{H}\cup J.
\]
 To conclude the proof of (2), we observe that
\[
\|a_{j}(t)\|_{2}^{2}=F_{\mu_{j}}'(\omega_{j}(t))-1,\quad j=1,2,t\in J,
\]
by Lemma \ref{lem:F'-1=00003DL2norm}.
\end{proof}
A converse of the above result holds as well.
\begin{prop}
\label{prop:if(F'-1)(F'-1)<1} Let $\mu_{1}$ and $\mu_{2}$ be Borel
probability measures on $\mathbb{R}$, neither of which is a point
mass, and set $\mu=\mu_{1}\boxplus\mu_{2}$. Suppose that the points
$t_{j}\in\mathbb{R}\backslash{\rm supp}(\mu_{j})$ are such that $G_{\mu_{1}}(t_{1})=G_{\mu_{2}}(t_{2})\ne0$
and $(F_{\mu_{1}}'(t_{1})-1)(F_{\mu_{2}}'(t_{2})-1)<1$. Then the
point
\[
t=t_{1}+t_{2}-F_{\mu_{1}}(t_{1})
\]
belongs to $\mathbb{R}\backslash{\rm supp}(\mu)$, $F_{\mu}(t)=F_{\mu_{j}}(t_{j}),$
and $\omega_{j}(t)=t_{j}$ for $j=1,2$.
\end{prop}

\begin{proof}
Recall that $\omega_{1}(t)$ is the Denjoy-Wolff point of the map
$\varphi_{t}:\mathbb{H}\to\mathbb{H}$ defined by
\[
\varphi_{t}(\lambda)=t+h_{2}(t+h_{1}(\lambda)),\quad\lambda\in\mathbb{H},
\]
where $h_{j}(\lambda)=F_{\mu_{j}}(\lambda)-\lambda$ for $j=1,2$.
Since $t_{j}\notin{\rm supp}(\mu_{j})$ and $G_{\mu_{j}}(t_{j})\ne0$,
the function $h_{j}$ continues analytically to a neighborhood of
$t_{j}$. We note first that $t_{1}$ is a fixed point for $\varphi_{t}$.
Indeed, 
\[
t+h_{1}(t_{1})=t+F_{\mu_{1}}(t_{1})-t_{1}=t_{2}
\]
by the definition of $t$, and 
\[
\varphi_{t}(t_{1})=t+h_{2}(t+h_{1}(t_{1}))=t+h_{2}(t_{2})=t_{1}
\]
by a similar calculation. The fact that $h_{j}$ continues analytically
to $t_{j}$ allows us to calculate the Carath\'eodory-Julia derivative
of $\varphi_{t}$ at the point $t_{1}$ as an ordinary derivative:
\begin{align*}
\varphi_{t}'(t_{1}) & =h_{2}'(t+h_{1}(t_{1}))h_{1}'(t_{1})=h_{2}'(t_{2})h_{1}'(t_{1})\\
 & =(F_{\mu_{2}}'(t_{2})-1)(F_{\mu_{1}}'(t_{1})-1).
\end{align*}
By hypothesis, this number is less than one, so we conclude that $\omega_{1}(t)=t_{1}$.
The equality $\omega_{2}(t)=t_{2}$ follows by switching the indices
$1$ and $2$ in the argument above.

In order to verify that $t\notin{\rm supp}(\mu)$, it is convenient
to consider freely independent selfadjoint random variables $x_{1},x_{2}\in\widetilde{\mathcal{A}}$
whose distributions are $\mu_{1}$ and $\mu_{2}$, respectively. Using
the notation in the proof of Proposition \ref{pro:omega maps resolvent to res},
we see that $\|a_{1}(t)\|_{2}\|a_{2}(t)\|_{2}<1$. The bounded invertibility
of $t1_{\mathcal{A}}-(x_{1}+x_{2})$ follows then from Remark \ref{rem:a bit of algebra}
and the proposition follows because ${\rm supp}(\mu)=\sigma(x_{1}+x_{2})$. 
\end{proof}
We summarize the results of Propositions \ref{pro:omega maps resolvent to res}
and \ref{prop:if(F'-1)(F'-1)<1} as follows.
\begin{thm}
\label{thm:resolvent of sum in terms of pairs}Let $\mu_{1}$ and
$\mu_{2}$ be Borel probability measures on $\mathbb{R}$, neither
of which is a point mass, and set $\mu=\mu_{1}\boxplus\mu_{2}$. Consider
the sets
\[
A_{\mu}=\{t\in\mathbb{R}\backslash{\rm supp}(\mu):G_{\mu}(t)\ne0\}
\]
and 
\begin{align*}
B_{\mu_{1},\mu_{2}}=\{(t_{1},t_{2})\in A_{\mu_{1}}\times A_{\mu_{2}}: & G_{\mu_{1}}(t_{1})=G_{\mu_{2}}(t_{2})\ne0\\
 & \text{ and }(F_{\mu_{1}}'(t_{1})-1)(F_{\mu_{2}}'(t_{2})-1)<1\}.
\end{align*}
Then the map $t\mapsto(\omega_{1}(t),\omega_{2}(t))$ is a homeomorphism
of $A_{\mu}$ onto $B_{\mu_{1},\mu_{2}}$ with inverse given by 
\[
(t_{1},t_{2})\mapsto t_{1}+t_{2}-F_{\mu_{1}}(t_{1})=t_{1}+t_{2}-F_{\mu_{2}}(t_{2}).
\]
 
\end{thm}

Suppose that $\nu$ is a Borel probability measure on $\mathbb{R}$
such that ${\rm supp}(\nu)$ has $N<+\infty$ components. Then the
complement $\mathbb{R}\backslash{\rm supp}(\nu)$ has at most $N+1$
connected components, all of them bounded with at most two exceptions.
The function $t\mapsto G_{\mu}(t)$ is decreasing on each of these
components and it is nonzero on the unbounded components---positive
on a component of the form $(\alpha,+\infty)$ and negative on a component
of the form $(-\infty,\alpha)$. Thus, $G_{\nu}$ has at most $N-1$
zeros, possibly one in each bounded component, and the set $A_{\nu}$
(defined in the statement of Theorem \ref{thm:resolvent of sum in terms of pairs})
has at most $2N$ components. On each of these components the function
$F_{\nu}$ is increasing and it does not change sign. More precisely,
there are at most $N$ components of $A_{\nu}$ on which $F_{\nu}$
is positive and at most $N$ components on which $F_{\nu}$ is negative.
Suppose now that $\mu_{1}$ and $\mu_{2}$ are two probability measures
on $\mathbb{R}$ such that ${\rm supp}(\mu_{j})$ has $n_{j}<+\infty$
connected components. Then the set
\[
C_{\mu_{1},\mu_{2}}=\{(t_{1},t_{2})\in A_{\mu_{1}}\times A_{\mu_{2}}:G_{\mu_{1}}(t_{1})=G_{\mu_{2}}(t_{2})\}
\]
 is a disjoint union of at most $2n_{1}n_{2}$ smooth curves. Indeed,
the equality $G_{\mu_{1}}(t_{1})=G_{\mu_{2}}(t_{2})$ implies in particular
that these two numbers must have the same sign. The set $B_{\mu_{1},\mu_{2}}$
is, of course, a relatively open subset of $C_{\mu_{1},\mu_{2}}$,
and this is easily analyzed when $n_{1}=n_{2}=1$.
\begin{thm}
\label{thm:connected supports}Let $\mu_{1}$ and $\mu_{2}$ be Borel
probability measures on $\mathbb{R}$, and set $\mu=\mu_{1}\boxplus\mu_{2}$.
If ${\rm supp}(\mu_{1})$ and ${\rm supp}(\mu_{2})$ are connected
then so is ${\rm supp}(\mu)$. If the support of either $\mu_{1}$
or $\mu_{2}$ is not bounded below \emph{(}respectively, above\emph{)},
then the support of $\mu$ is not bounded below \emph{(}respectively,
above\emph{).}
\end{thm}

\begin{proof}
The result is trivial if either $\mu_{1}$ or $\mu_{2}$ is a point
mass. Therefore, we may assume that this is not the case. For $j=1,2,$
denote by $A_{\mu_{j}}^{-}$ (respectively, $A_{\mu_{j}}^{+})$ the
(possibly empty) connected component of $A_{\mu_{j}}={\rm supp}(\mu_{j})$
on which $F_{\mu_{j}}$ is negative (respectively, positive). Using
the notation introduced in the proof of Theorem \ref{thm:resolvent of sum in terms of pairs},
\[
B_{\mu_{1},\mu_{2}}\subset(A_{\mu_{1}}^{-}\times A_{\mu_{2}}^{-})\cup(A_{\mu_{1}}^{+}\times A_{\mu_{2}}^{+}).
\]
We show that $B_{\mu_{1},\mu_{2}}\cap(A_{\mu_{1}}^{-}\times A_{\mu_{2}}^{-})$
is connected and nonempty if $A_{\mu_{1}}^{-}\times A_{\mu_{2}}^{-}\ne\varnothing$.
Thus, assume that $A_{\mu_{j}}^{-}=(-\infty,\alpha_{j})$ for some
$\alpha_{j}\in\mathbb{R}$, $j=1,2,$ and observe that the measure
$\sigma_{j}$ in the Nevanlinna representation
\[
F_{\mu_{j}}(\lambda)=c_{\mu_{j}}+\lambda+\int_{\mathbb{R}}\frac{1+\lambda t}{t-\lambda}\,d\sigma_{\mu_{j}}(t)
\]
is supported in $[\alpha_{j},+\infty)$. Therefore, $F_{\mu_{j}}$
continues analytically to $A_{\mu_{j}}^{-}$, and $F''_{\mu_{j}}>0$
on this interval. Therefore, $F'_{\mu_{j}}$ is increasing on $A_{\mu_{j}}^{-}$
and, obviously
\[
\lim_{t\downarrow-\infty}F'_{\mu_{j}}(t)=1.
\]
This ensures that $B_{\mu_{1},\mu_{2}}\cap(A_{\mu_{1}}^{-}\times A_{\mu_{2}}^{-})\ne\varnothing.$
Suppose now that $(s_{1},s_{2})$ and $(t_{1},t_{2})$ are two elements
of this intersection. For definiteness, we assume that $s_{1}<t_{1}.$
Since the function $F_{\mu_{1}}$ is increasing, we have
\[
F_{\mu_{2}}(t_{2})=F_{\mu_{1}}(t_{1})>F_{\mu_{1}}(s_{1})=F_{\mu_{2}}(s_{2}),
\]
 and thus $s_{2}<t_{2}$ because $F_{\mu_{2}}$ is increasing as well.
Thus $F_{\mu_{1}}$ and $F_{\mu_{2}}$ map the intervals $[s_{1},t_{1}]$
and $[s_{2},t_{2}]$ bijectively onto the same interval. Hence, the
set
\[
D=\{(\tau_{1},\tau_{2})\in A_{\mu_{1}}^{-}\times A_{\mu_{2}}^{-}:F_{\mu_{1}}(\tau_{1})=F_{\mu_{2}}(\tau_{2}),s_{2}\le\tau_{2}\le t_{2}\}
\]
can be described as the graph
\[
D=\{(\tau_{1},u(\tau_{1})):\tau_{1}\in[s_{1},t_{1}]\}
\]
of an increasing homeomorphism of $[s_{1},t_{1}]$ onto $[s_{2},t_{2}]$.
(Informally, we have $u=F_{\mu_{2}}^{-1}\circ F_{\mu_{1}}$.) Now,
we have
\[
(F'_{\mu_{1}}(\tau_{1})-1)(F'_{\mu_{2}}(u(\tau_{1}))-1)<1
\]
for $\tau_{1}\in\{s_{1},t_{1}\},$ and the fact that $F_{\mu_{j}}'$
is increasing implies that the inequality persists for every $\tau_{1}\in[s_{1},t_{1}]$.
This proves the connectedness of $B_{\mu_{1},\mu_{2}}\cap(A_{\mu_{1}}^{-}\times A_{\mu_{2}}^{-})$.
An analogous argument shows that $B_{\mu_{1},\mu_{2}}\cap(A_{\mu_{1}}^{+}\times A_{\mu_{2}}^{+})$
is connected as well. We conclude that $B_{\mu_{1},\mu_{2}}$ has
at most two connected components, and therefore the homeomorphic set
$A_{\mu}$ has the same property. It follows that ${\rm supp}(\mu)$
is connected.

If ${\rm supp}(\mu_{1})$ is not bounded below (respecively, above),
we have
\[
A_{\mu_{1}}^{-}\times A_{\mu_{2}}^{-}=\varnothing,
\]
(respectively, $A_{\mu_{1}}^{+}\times A_{\mu_{2}}^{+}=\varnothing$),
and therefore $\mathbb{R}\backslash{\rm supp}(\mu_{1}\boxplus\mu_{2})$
is bounded below (respectively, above). The theorem follows.
\end{proof}

\section{The analysis of free multiplicative convolution on $\mathbb{R}_{+}$\label{sec:mult-on-R}}

Given a Borel probability measure $\mu$ on $\mathbb{R}_{+}=[0,+\infty)$,
different from $\delta_{0}$, the function
\[
\psi_{\mu}(z)=\int_{\mathbb{R}_{+}}\frac{zt\,d\mu(t)}{1-zt},\quad z\in\mathbb{C}\backslash\mathbb{R}_{+},
\]
is analytic. We also consider the function
\[
\eta_{\mu}(z)=\frac{\psi_{\mu}(z)}{1+\psi_{\mu}(z)},\quad z\in\mathbb{C}\backslash\mathbb{R}_{+},
\]
and we denote by ${\rm Eta}_{\mathbb{R}_{+}}$the collection of all
functions of the form $\eta_{\mu}$. The function $\eta_{\mu}$ satisfies
the following properties:
\begin{enumerate}
\item $\eta_{\mu}(\mathbb{H})\subset\mathbb{H},$
\item $\eta_{\mu}((-\infty,0))\subset(-\infty,0),$
\item $\lim_{x\uparrow0}\eta_{\mu}(x)=0$, and
\item $\arg(\eta_{\mu}(z))\in[\arg(z),\pi]$ for every $z\in\mathbb{H}.$
(We use the principal value of the argument.)
\end{enumerate}
Conversely, every analytic function $\eta:\mathbb{C}\backslash\mathbb{R}_{+}\to\mathbb{C}$,
that satisfies these four properties is of the form $\eta=\eta_{\mu}$
for a unique probability measure $\mu\ne\delta_{0}$. (It is seen
in \cite{BPW trans} that (4) is a consequence of the other three
properties.) The uniqueness of $\mu$ can be seen, for instance, from
the identity
\[
\psi_{\mu}(z)+1=\frac{1}{z}G_{\mu}(\frac{1}{z}),\quad z\in\mathbb{C}\backslash\mathbb{R}.
\]
 This identity implies, in particular, that $\psi_{\mu}$ continues
analytically to all points $x\in(0,+\infty)$ such that the support
of $\mu$ does not contain $1/x$. For such values of $x$, we have
\[
\psi_{\mu}^{\prime}(x)=\int_{\mathbb{R}_{+}}\frac{t\,d\mu(t)}{(1-tx)^{2}},
\]
and this is strictly positive since $\mu\ne\delta_{0}$. Thus, if
$J\subset\mathbb{R}_{+}$ is an open interval on which $\psi_{\mu}$
continues analytically to a real-valued function such that $-1\notin\psi_{\mu}(J),$
the functions $\psi_{\mu}$ and $\eta_{\mu}$ are increasing on $J$. 

Given two Borel probability measures $\mu_{1},\mu_{2}$ on $\mathbb{R}_{+}$
their free multiplicative convolution $\mu_{1}\boxtimes\mu_{2}$ is
another such measure that is defined as follows. Suppose that $(\mathcal{A},\tau)$
is a tracial $W^{*}$-probability space, and $x_{1},x_{2}\in\widetilde{A}$
are two positive elements with distributions $\mu_{1}$ and $\mu_{2}$,
respectively. Then $\mu_{1}\boxtimes\mu_{2}$ is the distribution
of the positive operator $x_{2}^{1/2}x_{1}x_{2}^{1/2}$. It is easily
seen that 
\[
\psi_{\mu_{1}\boxtimes\mu_{2}}(z)=\tau((1_{\mathcal{A}}-zx_{1}x_{2})^{-1})-1,\quad z\in\mathbb{C}\backslash\mathbb{R}_{+}.
\]
The measure $\mu_{1}\boxtimes\mu_{2}$ can be obteined, just as in
the additive case, from the existence of appropriate subordination
functions.
\begin{thm}
\label{thm:LSubordination-half-line}Let $\mu_{1}$ and $\mu_{2}$
be Borel probability measures on $\mathbb{R}$, neither of which is
a point mass, and set $\mu=\mu_{1}\boxtimes\mu_{2}$. There exist
unique continuous functions $\omega_{1},\omega_{2}:\mathbb{H}\cup\mathbb{R}\to\mathbb{H}\cup\mathbb{R}\cup\{\infty\}$
such that\emph{:}
\begin{enumerate}
\item Each of the restrictions $\omega_{j}|\mathbb{H}$ is also the restriction
of a function in ${\rm Eta}_{\mathbb{R}_{+}}$.
\item $\eta_{\mu}(z)=\eta_{\mu_{1}}(\omega_{1}(z))=\eta_{\mu_{2}}(\omega_{2}(z))=\omega_{1}(z)\omega_{2}(z)/z$
for $z\in\mathbb{H}.$
\item Set $k_{j}(z)=\eta_{\mu_{j}}(z)/z$. Then, for every $t\in\mathbb{R}$,
$\omega_{1}(t)$ is the Denjoy-Wolff point of the map $\varphi_{t}(z)=tk_{2}(tk_{1}(z)),z\in\mathbb{H},$
and $\omega_{2}(t)$ is the Denjoy-Wolff point of the map $z\mapsto tk_{1}(tk_{2}(z)),z\in\mathbb{H}$.
\item For every $t\in\mathbb{H},$ the function $\varphi_{t}$ defined above
maps the domain
\[
\Omega_{t}=\{z\in\mathbb{H}:\arg(t)<\arg(z)<\pi\}
\]
into itself, and $\omega_{1}(t)$ is the Denjoy-Wolff point of the
map $\varphi_{t}|\Omega_{t}$. An analogous statement holds for $\omega_{2}(t)$.
\end{enumerate}
\end{thm}

\begin{rem}
\label{rem:the map varphi NOT the identity (R)}In case $t<0$, the
map $\varphi_{t}$ defined above sends the entire $\mathbb{C}\backslash\mathbb{R}_{+}$
to itself. Since $\omega_{2}(t)<0$ in this case, it follows that
$\omega_{2}(t)$ is an actual fixed point of $\varphi_{t}$. We observe
that the map $\varphi_{t}$ is not the identity map. Indeed, if this
were the case, the maps $z\mapsto-k_{j}(z)$, $j=1,2$, would have
to be conformal automorphisms of $\mathbb{C}\backslash\mathbb{R}_{+}.$
Since $k_{j}(\mathbb{R}_{-})\subset\mathbb{R}_{+}$, such conformal
maps would have to be of the form $-k_{j}(z)=\alpha z$ or $-k_{j}(z)=-\alpha/z$
for some $\alpha>0$. It is easy to verify that the maps $-\alpha z$
and $\alpha/z$ do not equal $\eta_{\mu}(z)/z$ for any measure $\mu$.
Since $\mathbb{C}\backslash\mathbb{R}_{+}$ is conformally equivalent
to the unit disk $\mathbb{D}$, the following result is a consequence
of this observation and of the Schwarz lemma.
\end{rem}

\begin{cor}
\label{cor:varphi'<1 at fixed point (R)}Under the hypotheses of Theorem\emph{
\ref{thm:LSubordination-half-line}}, we have $|\varphi_{t}'(\omega_{2}(t))|<1$
for every $t<0$. Equivalently,
\[
|t^{2}k_{1}'(\omega_{1}(t))k_{2}'(\omega_{2}(t))|<1,\quad t\in(-\infty,0).
\]
\end{cor}

We note for further use the calculation of $k_{j}'(\lambda)$. We
state it in a form in which it can also be applied to free multiplicative
convolution on the unit circle $\mathbb{T}$.
\begin{lem}
\label{lem:derivative of k}Given $x\in\widetilde{A}$, define analytic
functions
\[
\psi_{x}(\lambda)=\tau((1_{\mathcal{A}}-\lambda x)^{-1})-1,\quad\lambda\ne0,1/\lambda\in\mathbb{C}\backslash\sigma(x),
\]
\[
\eta_{x}(\lambda)=\frac{\psi_{x}(\lambda)}{1+\psi_{x}(\lambda)},\quad1/\lambda\in\mathbb{C}\backslash\sigma(x),\psi_{x}(\lambda)\ne-1,
\]
and 
\[
k_{x}(\lambda)=\frac{\eta_{x}(\lambda)}{\lambda},\quad1/\lambda\in\mathbb{C}\backslash\sigma(x),\psi_{x}(\lambda)\ne-1.
\]
Then
\[
k_{x}'(\lambda)=\frac{\tau((1_{\mathcal{A}}-\lambda x)^{-2})-\tau((1_{\mathcal{A}}-\lambda x)^{-1})^{2}}{\lambda^{2}(\psi_{x}(\lambda)+1)^{2}}.
\]
\end{lem}

\begin{proof}
Set $m_{1}=\tau((1_{\mathcal{A}}-\lambda x)^{-1})$ and $m_{2}=\tau((1_{\mathcal{A}}-\lambda x)^{-2})$,
and note that
\begin{align*}
\lambda\psi_{x}'(\lambda) & =\tau(\lambda x(1_{\mathcal{A}}-\lambda x)^{-2})\\
 & =\tau((1_{\mathcal{A}}-\lambda x)^{-2})-\tau((1_{\mathcal{A}}-\lambda x)^{-1})=m_{2}-m_{1},
\end{align*}
so
\begin{align*}
k_{x}'(\lambda) & =\frac{\psi_{x}'(\lambda)\lambda(\psi_{\mu}(\lambda)+1)-\psi_{x}(\lambda)[(\psi_{x}(\lambda)+1)+\lambda\psi'_{x}(\lambda)]}{\lambda^{2}(\psi_{x}(\lambda)+1)^{2}}\\
 & =\frac{\lambda\psi_{x}'(\lambda)-\psi_{x}(\lambda)(\psi_{x}(\lambda)+1)}{\lambda^{2}(\psi_{x}(\lambda)+1)^{2}}\\
 & =\frac{m_{2}-m_{1}-(m_{1}-1)m_{1}}{\lambda^{2}(\psi_{x}(\lambda)+1)^{2}}.
\end{align*}
The stated formula follows immediately.
\end{proof}
The following result is relevant in the calculation of $\varphi_{t}'$.
\begin{lem}
\label{lem:product of derivatives alpha eta etc}Suppose that $x_{1},x_{2}\in\widetilde{\mathcal{A}},$
$x=x_{1}x_{2},$and $\lambda_{1},\lambda_{2},\lambda,\,\alpha\in\mathbb{C}\backslash\{0\}$
satisfy the following conditions\emph{:}
\begin{enumerate}
\item $1/\lambda_{1}\notin\sigma(x_{1}),1/\lambda_{2}\notin\sigma(x_{2})$,
$1/\lambda\notin\sigma(x)$,
\item $\psi_{x_{1}}(\lambda_{1})=\psi_{x_{2}}(\lambda_{2})=\psi_{x}(\lambda)=\alpha\ne-1,$
and
\item $\lambda_{1}\lambda_{2}=\lambda\eta_{x}(\lambda)=\lambda\alpha/(1+\alpha)$.
\end{enumerate}
Then we have
\begin{align*}
k_{x_{1}}'(\lambda_{1})k_{x_{2}}'(\lambda_{2}) & =\frac{\tau((1_{\mathcal{A}}-\lambda_{1}x_{1})^{-2})-\tau((1_{\mathcal{A}}-\lambda_{1}x_{1})^{-1})^{2}}{\lambda^{2}\alpha^{2}(\alpha+1)^{2}}\\
 & \times(\tau((1_{\mathcal{A}}-\lambda_{2}x_{2})^{-2})-\tau((1_{\mathcal{A}}-\lambda_{2}x_{2})^{-1})^{2})
\end{align*}
\end{lem}

\begin{proof}
Two applications of Lemma \ref{lem:derivative of k} yield
\begin{align*}
k_{x_{1}}'(\lambda_{1})k_{x_{2}}'(\lambda_{2}) & =\frac{\tau((1_{\mathcal{A}}-\lambda_{1}x_{1})^{-2})-\tau((1_{\mathcal{A}}-\lambda_{1}x_{1})^{-1})^{2}}{\lambda_{1}^{2}(\alpha+1)^{2}}\\
 & \times\frac{\tau((1_{\mathcal{A}}-\lambda_{2}x_{2})^{-2})-\tau((1_{\mathcal{A}}-\lambda_{2}x_{2})^{-1})^{2}}{\lambda_{2}^{2}(\alpha+1)^{2}},
\end{align*}
and part (3) of the hypothesis allows us to replace the denominator
by
\[
\lambda_{1}^{2}(\alpha+1)^{2}\lambda_{2}^{2}(\alpha+1)^{2}=\lambda^{2}\alpha^{2}(\alpha+1)^{2},
\]
and thus obtain the formula in the statement.
\end{proof}
The following result first appeared in \cite[Section 3]{haagerup-fields}.
\begin{lem}
\label{lem:multiplicative Haagerup}Suppose that $y_{1},y_{2}\in\widetilde{A}$
are such that $1_{\mathcal{A}}-y_{1}$ and $1_{\mathcal{A}}-y_{2}$
are boundedly invertible, and let $\beta\in\mathbb{C}\backslash\{0,1\}$.
Then $1_{\mathcal{A}}-(\beta/(\beta-1))y_{1}y_{2}$ is boundedly invertible
if and only if
\[
1_{\mathcal{A}}-\frac{1}{\beta(\beta-1)}((1_{\mathcal{A}}-y_{1})^{-1}-\beta)((1_{\mathcal{A}}-y_{2})^{-1}-\beta)
\]
is invertible.
\end{lem}

\begin{proof}
The easily verified identity
\begin{align*}
(1_{\mathcal{A}}-y_{1}) & \left[1_{\mathcal{A}}-\frac{1}{\beta(\beta-1)}((1_{\mathcal{A}}-y_{1})^{-1}-\beta1_{\mathcal{A}})((1_{\mathcal{A}}-y_{2})^{-1}-\beta1_{\mathcal{A}})\right](1_{\mathcal{A}}-y_{2})\\
 & =\frac{1}{\beta}\left(1_{\mathcal{A}}-\frac{\beta}{\beta-1}y_{1}y_{2}\right)
\end{align*}
implies the stated result.
\end{proof}
We can now formulate a multiplicative analog of Theorem \ref{thm:resolvent of sum in terms of pairs}.
The functions $\omega_{j}$ in the statement are those provided by
Theorem \ref{thm:LSubordination-half-line}.
\begin{thm}
\label{thm:complement of supp multiplicative R}Let $\mu_{1}$ and
$\mu_{2}$ be nondegenerate Borel probability measures on $\mathbb{R}_{+}$,
and set $\mu=\mu_{1}\boxtimes\mu_{2}.$ If $t\in\mathbb{R}_{+}$ and
$1/t\notin{\rm supp}(\mu)$, denote by
\[
V_{\mu}(t)=\int_{\mathbb{R}_{+}}\frac{d\mu(s)}{(1-ts)^{2}}-\left[\int_{\mathbb{R}_{+}}\frac{d\mu(s)}{1-ts}\right]^{2}
\]
the variance of the function $s\mapsto1/(1-ts)$ in $L^{2}(\mu)$.
Define the following sets\emph{:
\begin{align*}
C_{\mu} & =\{t\in\mathbb{R}_{+}:1/t\notin{\rm supp}(\mu),\psi_{\mu}(t)\notin\{0,-1\}\},\\
D_{\mu_{1,}\mu_{2}} & =\{(t_{1},t_{2})\in C_{\mu_{1}}\times C_{\mu_{2}}:\psi_{\mu_{1}}(t_{1})=\psi_{\mu_{2}}(t_{2})\notin\{0,-1\},\text{\emph{and}}\\
 & V_{\mu_{1}}(t_{1})V_{\mu_{2}}(t_{2})/[\psi_{\mu_{1}}(t_{1})(\psi_{\mu_{1}}(t_{1})+1)]^{2}<1\}.
\end{align*}
}Then the function $t\mapsto(\omega_{1}(t),\omega_{2}(t))$ maps $C_{\mu}$
homeomorphically onto $D_{\mu_{1},\mu_{2}}$, and the inverse function
is given by
\[
(t_{1},t_{2})\mapsto\frac{t_{1}t_{2}}{\eta_{\mu_{1}}(t_{1})}=\frac{t_{1}t_{2}}{\eta_{\mu_{2}}(t_{2})}.
\]
\end{thm}

\begin{proof}
As in the case of additive free convolution, it is easy to verify
that $\omega_{j}$ is real-valued on the set $C_{\mu}$, and that
$1/\omega_{j}(t)\notin{\rm supp}(\mu_{j})$ for $t\in C_{\mu}$. Indeed,
the fact that $\arg\eta_{\mu}(z)=\arg\eta_{\mu_{j}}(\omega_{j}(z))\ge\arg\omega_{j}(z)$
implies that $\omega_{j}(t)$ is real for $t\in C_{\mu}$, and then
$\eta_{\mu}(t)=\eta_{\mu_{j}}(\omega_{j}(t))$ must be real as well
for such $t$. 

We consider freely independent nonnegative operators $x_{1},x_{2}\in\widetilde{\mathcal{A}}$
whose distributions are $\mu_{1}$ and $\mu_{2}$ respectively. Suppose
first that $(t_{1},t_{2})\in D_{\mu_{1},\mu_{2}}$ and define $t=t_{1}t_{2}/\eta_{\mu_{j}}(t_{j})$.
Observe that the functions $k_{j}$ of Theorem \ref{thm:LSubordination-half-line}(3)
can be identified as $k_{j}=k_{x_{j}}$ for $j=1,2$, and that $t_{1}$
is a fixed point of the map $\varphi_{t}$ defined there. Indeed,
\begin{align*}
\varphi_{t}(t_{1}) & =tk_{2}(tk_{1}(t_{1}))=tk_{2}\left(\frac{t\eta_{\mu_{1}}(t_{1})}{t_{1}}\right)\\
 & =tk_{2}(t_{2})=\frac{t\eta_{\mu_{2}}(t_{2})}{t_{2}}=t_{1}.
\end{align*}
Next, observe that
\[
\tau((1_{\mathcal{A}}-t_{j}x_{j})^{-2})-\tau((1_{\mathcal{A}}-t_{j}x_{j})^{-1})^{2}=V_{\mu_{j}}(t_{j}),\quad j=1,2,
\]
or, alternately,
\[
\|(1_{\mathcal{A}}-t_{j}x_{j})^{-1}-\tau((1_{\mathcal{A}}-t_{j}x_{j})^{-1})\|_{2}^{2}=V_{\mu_{j}}(t_{j}),\quad j=1,2.
\]
Setting $\alpha=\psi_{\mu_{1}}(t_{1})$, Lemmas \ref{lem:derivative of k}
and \ref{lem:product of derivatives alpha eta etc} (with $t,t_{1},t_{2}$
in place of $\lambda,\lambda_{1},\lambda_{2}$), and the definition
of $D_{\mu_{1},\mu_{2}},$ yield
\begin{align*}
\varphi_{t}'(t_{1}) & =t^{2}k_{x_{1}}'(t_{1})k_{x_{2}}'(t_{2})\\
 & =\frac{V_{\mu_{1}}(t_{1})V_{\mu_{2}}(t_{2})}{\alpha^{2}(1+\alpha)^{2}}<1,
\end{align*}
showing thereby that $t_{1}$ is the Denjoy-Wolff point of $\varphi_{t}$,
that is, $t_{1}=\omega_{1}(t)$. Similarly, $t_{2}=\omega_{2}(t)$.

Suppose now that $t\in C_{\mu}$, and consider the connected set $S=\mathbb{H}\cup(-\infty,0)\cup\{t\}$.
For every $\lambda\in S$, the operators $1_{\mathcal{A}}-\lambda x_{1}x_{2},$
$1_{\mathcal{A}}-\omega_{1}(\lambda)x_{1}$, and $1_{\mathcal{A}}-\omega_{2}(\lambda)x_{2}$
are boundedly invertible, and 
\begin{align*}
\tau((1_{\mathcal{A}}-\lambda x_{1}x_{2})^{-1}) & =\tau((1_{\mathcal{A}}-\omega_{1}(\lambda)x_{1})^{-1}\\
 & =\tau((1_{\mathcal{A}}-\omega_{2}(\lambda)x_{2})^{-1}=\psi_{\mu}(\lambda)+1.
\end{align*}
Setting $y_{1}=\omega_{1}(\lambda)x_{1},$ $y_{2}=\omega_{2}(\lambda)x_{2}$,
and $\beta=\psi_{\mu}(\lambda)+1$ in Lemma \ref{lem:multiplicative Haagerup},
we conclude that
\[
1_{\mathcal{A}}-\frac{((1_{\mathcal{A}}-\omega_{1}(\lambda)x_{1})^{-1}-(\psi_{\mu}(\lambda)+1)1_{\mathcal{A}})((1_{\mathcal{A}}-\omega_{2}(\lambda)x_{2})^{-1}-(\psi_{\mu}(\lambda)+1)1_{\mathcal{A}})}{(\psi_{\mu}(\lambda)+1)\psi_{\mu}(\lambda)}
\]
 is invertible for every $\lambda\in S$. Since $\psi_{\mu}(\lambda)+1=\tau((1_{\mathcal{A}}-\omega_{1}(\lambda)x_{1})^{-1})$,
Corollary \ref{cor:varphi'<1 at fixed point (R)} and Lemma \ref{lem:product of derivatives alpha eta etc}
imply that
\begin{align*}
 & \left|\frac{1}{(\psi_{\mu}(\lambda)+1)\psi_{\mu}(\lambda)}\right|\|(1_{\mathcal{A}}-\omega_{1}(\lambda)x_{1})^{-1}-(\psi_{\mu}(\lambda)+1)1_{\mathcal{A}}\|_{2}\\
 & \times\|(1_{\mathcal{A}}-\omega_{2}(\lambda)x_{2})^{-1}-(\psi_{\mu}(\lambda)+1)1_{\mathcal{A}}\|_{2}<1
\end{align*}
for $\lambda\in(-\infty,0)$. By Lemma \ref{lem:Lehner}, this inequality
is also true for every $\lambda\in S$, in particular for $\lambda=t$.
We conclude that the pair $(\omega_{1}(t),\omega_{2}(t))$ belongs
to $D_{\mu_{1},\mu_{2}}$.
\end{proof}
Before discussing the connectivity of the support of a free multiplicative
convolution, we require a few observations about the functions $\psi_{\mu}$
and $\eta_{\mu}$ associated with a nondegenerate probability measure
on $\mathbb{R}_{+}$. We recall the identity
\[
\psi_{\mu}(z)+1=\frac{1}{z}G_{\mu}\left(\frac{1}{z}\right),\quad z\in\mathbb{C}\backslash{\rm supp}(\mu),
\]
 that implies, in particular, that $\psi_{\mu}$ and $\eta_{\mu}$
continue analytically (meromorphically when $\psi_{\mu}(x)=-1$) to
all points $x\in(0,+\infty)$ such that $1/x\notin{\rm supp}(\mu)$.
For such values of $x$, we have
\[
\psi_{\mu}^{\prime}(x)=\int_{\mathbb{R}_{+}}\frac{t\,d\mu(t)}{(1-tx)^{2}},
\]
and this is strictly positive since $\mu\ne\delta_{0}$. Thus, if
$J\subset\mathbb{R}_{+}$ is an open interval on which $\psi_{\mu}$
continues analytically to a real-valued function such that $-1\notin\psi_{\mu}(J),$
the functions $\psi_{\mu}$ and $\eta_{\mu}$ are increasing on $J$.
In particular, $\psi_{\mu}$ and $\eta_{\mu}$ can only change sign
once in such an interval. The inequalities $\arg(z)\le\arg(\eta_{\mu}(z))\le\pi$
impliy that one can write
\[
\eta_{\mu}(z)=z\exp(u_{\mu}(z)),\quad z\in\mathbb{H}\cup(\mathbb{R}\backslash{\rm supp}(\mu)),
\]
 where $u$ is continuous, analytic on $\mathbb{H},$ $u>0$ on $(-\infty,0)$,
and $0\le\Im u(z)\le\pi$. In particular, $u_{\mu}$ has a Nevanlinna
representation of the form
\[
u_{\mu}(z)=a_{\mu}+\int_{\mathbb{R}_{+}}\frac{1+zs}{s-z}\,d\rho_{\mu}(s),\quad z\in\mathbb{H}\cup(-\infty,0),
\]
for some $a_{\mu}\in\mathbb{R}$ and some finite Borel measure $\rho_{\mu}$
on $\mathbb{R}_{+}$. (There is no term of the form $bz,b>0,$ in
this representation because the imaginary part of $u_{\mu}$ is bounded.)

If $J$ is an open interval where $\psi_{\mu}$ has a real-valued
continuation such that $0,-1\notin\psi_{\mu}(J)$, the function $u_{\mu}$
continues analytically to $J$, it takes real values if $\psi_{\mu}(J)\subset(0,+\infty)$,
and it has imaginary part $\pi$ if $\psi_{\mu}(J)\subset(-\infty,0)$.
In particular, if $\psi_{\mu}(J)\subset(0,+\infty)$, then $J$ is
disjoint from the support of $\rho_{\mu}$.
\begin{lem}
\label{lem:u_mu when supp connected}With the notation above, suppose
that ${\rm supp(\mu)\subset[\alpha,\beta]},$where $0<\alpha<\beta<+\infty$.
Then\emph{:}
\begin{enumerate}
\item We have $\psi_{\mu}(t)<-1$ for $t\in(1/\alpha,+\infty)$. 
\item We have $\psi_{\mu}(t)>0$ and $u_{\mu}(t)<0$ for $t\in(0,1/\beta)$.
\item ${\rm supp}(\rho_{\mu})\subset[1/\beta,1/\alpha]$.
\item The function 
\[
v_{\mu}(z)=u_{\mu}(e^{z})
\]
 is convex on $(-\log(\alpha),+\infty)$ and concave on $(-\infty,-\log(\beta)).$
\end{enumerate}
Properly modified statements are true when ${\rm supp}(\mu)$ is of
the form $[\alpha,+\infty)$ or $[0,\beta]$ for some $\alpha\ge0,\beta>0$.
In these cases, one or both of the intervals $(0,1/\beta)$ and $(1/\alpha,+\infty)$
must be replaced by $\varnothing$.

\end{lem}

\begin{proof}
To prove (1), we note that for $t>1/\alpha$,
\[
\frac{ts}{1-ts}<-1,\quad s\in[\alpha,\beta],
\]
integration in $s$ yields $\psi_{\mu}(t)<-1$, and hence $\eta_{\mu}(t)>0$.
For $t<1/\alpha$,
\[
\frac{ts}{1-ts}>0,\quad s\in[\alpha,\beta],
\]
so $\psi_{\mu}(t),\eta_{\mu}(t)$, and $u_{\mu}(t)$ are all positive.
This verifies (2). 

Statements (1) and (2) imply that $u$ is real-valued on $(-\infty,1/\beta)\cup(1/\alpha,+\infty)$,
and (3) follows via Stieltjes inversion. Finally, a straightforward
calculation yields
\[
v_{\mu}''(z)=2\int_{1/\beta}^{1/\alpha}\frac{(2+s)e^{z}}{(s-e^{z})^{3}}\,d\rho_{\mu}(s),\quad z\notin[1/\beta,1/\alpha].
\]
This expression is negative when $e^{z}\in(1/\alpha,+\infty)$ and
positve when $e^{z}\in(0,1/\beta)$, thus proving (4).
\end{proof}
We can now prove an analog of Theorem \ref{thm:connected supports}.
\begin{thm}
\label{thm:connected support R+}Let $\mu_{1}$ and $\mu_{2}$ be
two probability measures on $\mathbb{R}_{+},$and set $\mu=\mu_{1}\boxtimes\mu_{2}$.
If ${\rm supp(\mu_{1})}$ and ${\rm supp}(\mu_{2})$ are connected,
then ${\rm supp}(\mu)$ is connected as well. If either ${\rm supp(}\mu_{1})$
or ${\rm supp}(\mu_{2})$ contains $0$ \emph{(}respectively, is unbounded\emph{)}
then ${\rm supp}(\mu)$ contains $0$ \emph{(}respectively, is unbounded\emph{).}
\end{thm}

\begin{proof}
If one of the measures $\mu_{j}$ is degenerate, the result is trivial.
Therefore we may, and do, assume that the measures are nondegenerate.
In this case, the set
\[
E_{\mu_{1},\mu_{2}}=\{(t_{1},t_{2})\in(0,+\infty):\psi_{\mu_{1}}(t_{1})=\psi_{\mu_{2}}(t_{2})\}
\]
is the union of (at most) two connected curves, on one of which $\psi_{\mu_{1}}(t_{1})$
is positive, and on the other $\psi_{\mu_{1}}(t_{1})<-1$. This follows
from the fact that $\psi_{\mu_{j}}$ is increasing on every interval
disjoint from ${\rm supp}(\mu_{j})$. The set $E_{\mu_{1},\mu_{2}}$
is the graph of the increasing function $h=\psi_{\mu_{2}}^{-1}\circ\psi_{\mu_{1}}$
defined on the union of two intervals $J_{1}\subset(-\infty,-1)$
and $J_{2}\subset(0,+\infty)$. Thus, it suffices to show that the
intersection of each of these curves with the set $D_{\mu_{1},\mu_{2}}$
(defined in Theorem \ref{thm:complement of supp multiplicative R})
is connected. To do this, we observe that for $t>0,$ the equation
$\omega_{1}(t)=\varphi_{t}(\omega_{1}(t))$ can be written as
\[
\log(\omega_{1}(t))=\log t+v_{\mu_{2}}(\log t+v_{\mu_{1}}(\log(\omega_{1}(t))),
\]
where $v_{\mu_{j}}$ are as in Lemma \ref{lem:u_mu when supp connected}.
In other words, $\log(\omega_{1}(t))$ is the Denjoy-Wolff point of
the analytic selfmap
\[
\psi_{t}(z)=\log t+v_{\mu_{2}}(\log t+v_{\mu_{1}}(z))
\]
of the strip $\mathbb{S}=\{x+iy:0<y<\pi\}$. Moreover, the (Carath\'eodory-Julia)
derivative of $\psi_{t}$ at the point $\log(\omega_{1}(t))$ is 
\begin{align*}
\psi_{t}'(\log(\omega_{1}(t)) & =v_{\mu_{2}}'(\log t+v_{\mu_{1}}(\log(\omega_{1}(t)))v_{\mu_{1}}'(\log(\omega_{\mu_{1}}(t))\\
 & =v_{\mu_{2}}'(\log(\omega_{2}(t)))v_{\mu_{1}}'(\log(\omega_{\mu_{1}}(t)),
\end{align*}
and this quantity is less than one when $t\in C_{\mu}$. More generally,
if $(t_{1},t_{2})\in E_{\mu_{1},\mu_{2}}$ and we set $t=t_{1}t_{2}/\psi_{\mu_{1}}(t_{1})$,
it is easy to verify that $\log(t_{1})$ is a fixed point for $\psi_{t}$,
and
\[
\psi_{t}'(\log(t_{1}))=v_{\mu_{2}}'(\log(t_{2}))v_{\mu_{1}}'(\log(t_{1}))=v_{\mu_{2}}'(\log(h(t_{1})))v_{\mu_{1}}'(\log(t_{1})),\quad t_{1}\in J_{1}\cup J_{2}.
\]
Lemma \ref{lem:u_mu when supp connected}(4) shows that the above
expression is a monotone function of $t_{1}$ on each of the two components
of $E_{\mu_{1},\mu_{2}}$. In particular, the inequality $\psi_{t}'(\log(t_{1}))<1$
is satisfied in a connected subset of each of the intervals $J_{1}$
and $J_{2}$. The connectedness of ${\rm supp}(\mu)$ follows because
its complement has at most two connected components. If either $0\in{\rm supp}(\mu_{1})$
or $\sup{\rm supp(\mu_{1})=+\infty}$, then $E_{\mu_{1},\mu_{2}}$
consists of one connected component (depending on the sign of $\psi_{\mu_{1}}$
on the remaining---if any---component of $\mathbb{R}_{+}\backslash{\rm supp}(\mu_{1})$),
and the concluding statement of the theorem is easily derived from
this fact.
\end{proof}

\section{\label{sec:multiplicative-T}Multiplicative convolution on $\mathbb{T}$ }

The free multiplicative convolution of probability measures on the
unit circle is calculated using the same analytical apparatus as that
in Section \ref{sec:mult-on-R}, but the domain of analyticity is
the unit disk $\mathbb{D}$. Suppose that $\mu$ is such a measure.
Then we consider the analytic functions
\[
\psi_{\mu}(z)=\int_{\mathbb{T}}\frac{zt}{1-zt}\,d\mu(t),\quad z\in\mathbb{D},
\]
and
\[
\eta(z)=\frac{\psi_{\mu}(z)}{1+\psi_{\mu}(z)},\quad z\in\mathbb{D}.
\]
Since $\Re(\lambda/(1-\lambda))>-1/2$ for $\lambda\in\mathbb{D},$we
have $\Re\psi_{\mu}(z)>1/2$ for every $z\in\mathbb{D}$, and therefore
$\eta_{\mu}$ maps $\mathbb{D}$ to itself. Obviously, $\psi_{\mu}(0)=\eta_{\mu}(0)=0.$
The functions $\psi_{\mu}$ and $\eta_{\mu}$ continue analytically
to any point $z\in\mathbb{T}$ such that $1/z\notin{\rm supp}(\mu)$
and, moreover,
\[
\Re\psi_{\mu}(z)=-\frac{1}{2}\text{ and }|\eta_{\mu}(z)|=1,\quad z\in\mathbb{T},\frac{1}{z}\notin{\rm supp}(\mu).
\]

The free convolution of two measures $\mu_{1},\mu_{2}$ on $\mathbb{T}$,
identified as the distributions of {*}-free unitary operators $x_{1},x_{2}$
in a tracial $W^{*}$-probability space $(\mathcal{A},\tau)$, is
the distribution of $x_{1}x_{2}$. There are again subordination functions,
analogous to those in Theorems \ref{thm:boxplus subordination} and
\ref{thm:LSubordination-half-line}.
\begin{thm}
\label{thm:LSubordination--circle}Let $\mu_{1}$ and $\mu_{2}$ be
Borel probability measures on $\mathbb{T}$, neither of which is a
point mass, and set $\mu=\mu_{1}\boxtimes\mu_{2}$. There exist unique
continuous functions $\omega_{1},\omega_{2}:\mathbb{D}\cup\mathbb{T}\to\mathbb{D}\cup\mathbb{T}$
such that\emph{:}
\begin{enumerate}
\item Each of the restrictions $\omega_{j}|\mathbb{D}$ is analytic and
$\omega_{j}(0)=0$.
\item $\eta_{\mu}(z)=\eta_{\mu_{1}}(\omega_{1}(z))=\eta_{\mu_{2}}(\omega_{2}(z))=\omega_{1}(z)\omega_{2}(z)/z$
for $z\in\mathbb{D}.$
\item Set $k_{j}(z)=\eta_{\mu_{j}}(z)/z$. Then, for every $t\in\mathbb{D}\cup\mathbb{T}$,
$\omega_{1}(t)$ is the Denjoy-Wolff point of the map $\varphi_{t}(z)=tk_{2}(tk_{1}(z)),z\in\mathbb{D},$
and $\omega_{2}(t)$ is the Denjoy-Wolff point of the map $z\mapsto tk_{1}(tk_{2}(z)),z\in\mathbb{D}$.
\end{enumerate}
\end{thm}

An analog of the following result was used in the proof of Theorem
\ref{thm:complement of supp multiplicative R}.
\begin{lem}
\label{lemma: variance circle}Suppose that $x\in\mathcal{A}$ is
a unitary operator with distribution $\mu$. Then
\[
\|(1-zx)^{-1}-\tau((1-zx)^{-1})\|_{2}^{2}=|\tau((1-zx)^{-2})-\tau((1-zx)^{-1})^{2}|
\]
for every $z\in\mathbb{T}$ such that $1/z\notin{\rm supp}(\mu)$
.
\end{lem}

\begin{proof}
We have
\[
\Re\frac{1}{1-tz}=\Re(\psi_{\mu}(z)+1)=\frac{1}{2},\quad t\in{\rm supp}(\mu).
\]
 Now, 
\[
\|(1-zx)^{-1}-\tau((1-zx)^{-1})\|_{2}^{2}=\int_{\mathbb{T}}\left|\frac{1}{1-zt}-\int_{\mathbb{T}}\frac{1}{1-zt}\,d\mu(t)\right|^{2}d\mu(t),
\]
and the integrand is purely imaginary. Thus
\[
\|(1-zx)^{-1}-\tau((1-zx)^{-1})\|_{2}^{2}=-\int_{\mathbb{T}}\left[\frac{1}{1-zt}-\int_{\mathbb{T}}\frac{1}{1-zt}\,d\mu(t)\right]^{2}d\mu(t),
\]
which leads immediately to the identity in the statement.
\end{proof}
Another ingredient in the proof of Theorem \ref{thm:complement of supp multiplicative R}
is the fact that
\begin{align*}
\left|\frac{1}{(\psi_{\mu}(\lambda)+1)\psi_{\mu}(\lambda)}\right|\|(1_{\mathcal{A}}-\omega_{1}(\lambda)x_{1})^{-1}-(\psi_{\mu}(\lambda)+1)1_{\mathcal{A}}\|_{2}\\
\times\|(1_{\mathcal{A}}-\omega_{1}(\lambda)x_{1})^{-1}-(\psi_{\mu}(\lambda)+1)1_{\mathcal{A}}\|_{2}<1
\end{align*}
for some $\lambda$. When $x_{1}$ and $x_{2}$ are unitary operators,
this inequality is satisfied for $\lambda=0$. Thus, all ingredients
necessary for the proof of the following result are in place since
Lemmas \ref{lem:derivative of k} and \ref{lem:product of derivatives alpha eta etc}
were proved for arbitrary elements in $\widetilde{\mathcal{A}}.$
The reader will have no difficulty in producing the necessary details.
\begin{thm}
\label{thm:complement of supp multiplicative T}Let $\mu_{1}$ and
$\mu_{2}$ be nondegenerate Borel probability measures on $\mathbb{T}$,
and set $\mu=\mu_{1}\boxtimes\mu_{2}.$ If $t\in\mathbb{T}$ and $1/t\notin{\rm supp}(\mu)$,
denote by
\[
V_{\mu}(t)=\left|\int_{\mathbb{R}_{+}}\frac{d\mu(s)}{(1-ts)^{2}}-\left[\int_{\mathbb{R}_{+}}\frac{d\mu(s)}{1-ts}\right]^{2}\right|
\]
the variance of the function $s\mapsto1/(1-ts)$ in $L^{2}(\mu)$.
Define the following sets\emph{:
\begin{align*}
C_{\mu} & =\{t\in\mathbb{T}:1/t\notin{\rm supp}(\mu)\},\\
D_{\mu_{1,}\mu_{2}} & =\{(t_{1},t_{2})\in C_{\mu_{1}}\times C_{\mu_{2}}:\psi_{\mu_{1}}(t_{1})=\psi_{\mu_{2}}(t_{2}),\text{\emph{and}}\\
 & V_{\mu_{1}}(t_{1})V_{\mu_{2}}(t_{2})/|\psi_{\mu_{1}}(t_{1})(\psi_{\mu_{1}}(t_{1})+1)|^{2}<1\}.
\end{align*}
}Then the function $t\mapsto(\omega_{1}(t),\omega_{2}(t))$ maps $C_{\mu}$
homeomorphically onto $D_{\mu_{1},\mu_{2}}$, and the inverse function
is given by
\[
(t_{1},t_{2})\mapsto\frac{t_{1}t_{2}}{\eta_{\mu_{1}}(t_{1})}=\frac{t_{1}t_{2}}{\eta_{\mu_{2}}(t_{2})}.
\]
 
\end{thm}

The map $\eta_{\mu}$ is orientation preserving on each connected
component of the set $\{t\in\mathbb{T}:1/t\notin{\rm supp}(\mu)\}$.
If ${\rm supp}(\mu_{1})$ has $n_{1}$ components and ${\rm supp}(\mu_{2})$
has $n_{2}$ components, the set
\[
E_{\mu_{1},\mu_{2}}=\{(t_{1},t_{2})\in C_{\mu_{1}}\times C_{\mu_{2}}:\psi_{\mu_{1}}(t_{1})=\psi_{\mu_{2}}(t_{2})\}
\]
has at most $n_{1}n_{2}$ components. On the other hand, ${\rm supp}(\mu)$
has as many components as the set $D_{\mu_{1},\mu_{2}}$. Thus, to
estimate the latter number, we need to understand how many components
of $D_{\mu_{1},\mu_{2}}$ are contained in each component of $E_{\mu_{1},\mu_{2}}$.
The existence of such estimates remains an open problem for now.

\end{document}